\documentclass[11pt,a4paper]{amsart}

\usepackage[T1]{fontenc}
\usepackage[UKenglish]{babel}

\usepackage{amssymb,amsfonts,amsmath,amsthm,mathrsfs,mathtools}

\newtheorem{thm}{Theorem}[section]
\newtheorem*{thm*}{Theorem}
\newtheorem{lma}[thm]{Lemma}

\newtheorem*{qu*}{Question}

\newtheorem{prop}[thm]{Proposition}

\theoremstyle{definition}
\newtheorem{remark}[thm]{Remark}

\numberwithin{equation}{section}		
\numberwithin{figure}{section}		

\usepackage{enumitem}  

\setlist[enumerate]{itemsep=1.5ex,topsep=1.5ex}

\setlist[enumerate,1]{label=(\roman*)}

\newcommand{\harmsp}{\mathcal{H}(\mathbb{R}^N)}

\newcommand{\harmpolysArgs}[2]{\mathcal{H}_{#1}(\mathbb{R}^{#2})}

\newcommand{\harmpolys}[1]{\harmpolysArgs{#1}{N}}

\newcommand{\Rplus}{\mathbb{R}_+}

\DeclarePairedDelimiter\floor{\lfloor}{\rfloor}

\author{Clifford Gilmore}
\address{Department of Mathematics and Statistics\\ P.O. Box 68\\ FI-00014 University of Helsinki\\ Finland.}
\email{clifford.gilmore@helsinki.fi}

\author{Eero Saksman}
\address{Department of Mathematics and Statistics\\ P.O. Box 68\\ FI-00014 University of Helsinki\\ Finland.}
\email{eero.saksman@helsinki.fi}

\author{Hans-Olav Tylli}
\address{Department of Mathematics and Statistics\\ P.O. Box 68\\ FI-00014 University of Helsinki\\ Finland.}
\email{hans-olav.tylli@helsinki.fi}

\thanks{The first and second authors have been supported by the Academy of Finland via the Centre of Excellence in Analysis and Dynamics Research (project no. 271983).  The first author has also been supported by the Doctoral Programme in Mathematics and Statistics of the University of Helsinki.}

\title[Growth of frequently hypercyclic harmonic functions]{Optimal growth of harmonic functions frequently hypercyclic for the partial differentiation operator}
\date{\today}

\begin{document}

\begin{abstract}
We solve a problem posed by Blasco, Bonilla and Grosse-Erdmann in 2010
by constructing a harmonic function on $\mathbb{R}^N$, that is frequently hypercyclic with respect to the partial differentiation operator $\partial/\partial x_k$ and which has a minimal growth rate in terms of the average $L^2$-norm  on spheres of radius $r>0$ as $r \to \infty$.
\end{abstract}

\maketitle

\section{Introduction}

For a separable  Fr\'{e}chet space $X$, the continuous linear operator $T \colon X \to X$ 
is  \emph{hypercyclic}  if there exists  $x \in X$ (called a \emph{hypercyclic vector})
such that its orbit under $T$ is dense in $X$, that is
\begin{equation*}
\overline{\{T^n x : n \geq 0\}} = X.
\end{equation*}
A stronger property was introduced by Bayart and Grivaux in \cite{BG06}, where they defined $T \colon X \to X$ to be \emph{frequently hypercyclic} if there exists $x \in X$ such that for any nonempty open set $U \subset X$ one has
\begin{equation*}
\liminf_{m\to\infty} \frac{\# \{ n : T^n x \in U,\; 0 \leq n \leq m\}}{m}  > 0.
\end{equation*}
Here $\#$ denotes the cardinality of the set.  The definition states that the set of iterations, for which the orbit of $x$ visits any given neighbourhood of $X$, has positive lower density and
 such an $x \in X$ is  called a \emph{frequently hypercyclic vector} for $T$.
Comprehensive introductions to the area of linear dynamics can be found in \cite{BM09} and \cite{GEP11}.

It was shown in \cite[Example 2.4]{BG06} that the differentiation operator $f \mapsto f'$
is frequently hypercyclic on the  space of entire holomorphic functions on $\mathbb{C}$
and estimates for the growth  of its frequently hypercyclic vectors, in terms of average $L^p$-norms on spheres of radius $r>0$ as $r \to \infty$, were found by Blasco, Bonilla and 
Grosse-Erdmann~\cite[Theorems 2.3, 2.4]{BBGE10} and 
Bonet and Bonilla~\cite[Corollary 2.4]{BB13}.  
The minimal growth rates were subsequently established by Drasin and Saksman~\cite{DS12}.

 Aldred and Armitage~\cite{AA98a} previously identified sharp growth rates,  in terms of the average $L^2$-norm on spheres of radius $r>0$ as $r \to \infty$, of hypercyclic vectors for the partial differentiation operators  on the space $ \harmsp$ of  harmonic functions  on $\mathbb{R}^N$, 
which we denote by
\begin{equation*}
\partial/\partial x_k \colon \harmsp \to \harmsp
\end{equation*}
for $N \geq 2$ and $1 \leq k \leq N$.

Subsequently Blasco et al.~\cite[Theorem 4.2]{BBGE10} computed growth rates, again in terms of the $L^2$-norm  on spheres of radius $r>0$,  in the frequently hypercyclic case and they asked  about the minimal growth rates of frequently hypercyclic vectors for $\partial/\partial x_k$ on $\harmsp$.  In this paper we answer their question 
by explicitly constructing a frequently hypercyclic harmonic function with the prescribed growth rate.

\section{Frequent Hypercyclicity of the Partial Differentiation Operator} \label{sec:background}

In this section we recall the exact question posed by Blasco et al.~\cite{BBGE10}
and state our main result.
We first introduce the notions and background required to discuss the problem in precise
terms.

Denote by  $S(r)$  the sphere of radius $r$ in the euclidean metric $\vert \, \cdot \, \vert$
 centred at the origin of $\mathbb{R}^N$ and  let $\sigma_r$  be the normalised $(N-1)$-dimensional measure on $S(r)$, so that $\sigma_r(S(r)) = 1$.  
For $h \in \harmsp$ and $r > 0$ we let
\begin{equation}  \label{defn:M_2}
M_2 (h, r) = \left( \int_{S(r)} | h |^2 \; \mathrm{d}\sigma_r \right)^{1/2}
\end{equation}
denote the $2$-integral mean of $h$ on $S(r)$ and 
for $g, h \in \harmsp$ the corresponding inner product  is written as
\begin{equation*}
\langle g, h \rangle_r = \int_{S(r)} gh \; \mathrm{d}\sigma_r. 
\end{equation*}
The space $\harmsp$ of harmonic functions is a Fr\'{e}chet space when equipped 
 with the complete metric
\begin{equation*}
d(g,h) = \sum_{n=1}^\infty 2^{-n} \frac{\lvert g-h \rvert_{S(n)}}{1 + \lvert g-h \rvert_{S(n)}} 
\end{equation*}
for $g,h \in \harmsp$ and it corresponds to the topology of local uniform convergence.
 Above we set $\vert f\vert_{S(n)} = \sup_{\vert x\vert = n} \vert f(x)\vert$
for $f \in \harmsp$.

Aldred and Armitage~\cite[Theorem~1]{AA98a} proved that given any function 
$\varphi \colon \Rplus \to \Rplus$ with $\varphi(r) \to \infty$ as $r \to \infty$, 
there exists a harmonic function $h \in \harmsp$ which is a 
$\partial/\partial x_k$-hypercyclic vector,  for $1 \leq k \leq N$, such that 
\begin{equation*}
M_2(h,r) \leq \varphi(r) \frac{e^r}{r^{(N-1)/2}}
\end{equation*}
for $r > 0$ sufficiently large. 
Furthermore,  they showed  there does not exist a $\partial/\partial x_k$-hypercyclic vector $h \in \harmsp$ that satisfies
\begin{equation}\label{hccase}
M_2(h,r) \leq C \frac{e^r}{r^{(N-1)/2}}
\end{equation}
for $r > 0$ and any constant $C > 0$.
Strictly speaking, the results in \cite{AA98a} are stated for the more general concept of universality of the family $\{D^{\alpha}: \alpha \in \mathbb{N}^N\}$ of all partial derivatives and the preceding hypercyclicity growth results for $\partial/\partial x_k$ are implicit in the proofs.

Subsequently Blasco et al.~\cite[Section 4]{BBGE10} considered the frequently hypercyclic case, where  they obtained the following  $L^2$-growth rates  for  $1 \leq k \leq N$.
\begin{enumerate}[label=\arabic*.,leftmargin=1.5em, rightmargin=0.5em] 
\item Let $\varphi \colon \Rplus \to \Rplus$ be any function with $\varphi(r) \to \infty$ as $r \to \infty$.  Then there exists a $\partial/\partial x_k$-frequently hypercyclic function $h \in \harmsp$ with
\begin{equation*}
M_2(h,r) \leq \varphi(r) \frac{e^r}{r^{N/2 - 3/4}}
\end{equation*}
for $r > 0$ sufficiently large.

\item Let $\psi \colon \Rplus \to \Rplus$ be any function with $\psi(r) \to 0$ as $r \to \infty$.  Then there is no  $\partial/\partial x_k$-frequently hypercyclic vector $h \in \harmsp$ with
\begin{equation*}
M_2(h,r) \leq \psi(r) \frac{e^r}{r^{N/2 - 3/4}}
\end{equation*}
for $r > 0$ sufficiently large.
\end{enumerate}
Moreover, they  asked~\cite[Section 6]{BBGE10} whether there exists a  
$\partial/\partial x_k$-frequently hypercyclic vector $h \in \harmsp$ such that the above function 
$\varphi$  can be replaced with a  constant in the growth rate.

We answer this question in the positive in the following theorem, 
using a modification of the approach of
Drasin and Saksman \cite{DS12} in the case of the entire functions.
\begin{thm} \label{thm:growthRate}
Let $N \geq 2$ and $1 \leq k \leq N$.  Then for any constant $C > 0$ there exists 
a $\partial/\partial x_k$-frequently hypercyclic harmonic function $h \in \harmsp$ such that 
\begin{equation}\label{eq:minGro}
M_2(h,r) \leq C \frac{e^r}{r^{N/2 - 3/4}}
\end{equation}
for all $r > 0$. 
\end{thm}
Similar to   \cite{DS12}, the argument involves the explicit  construction of a  function in $\harmsp$ that is a frequently hypercyclic vector for $\partial/\partial x_k$. By contrast, \cite{BBGE10} applies 
a generalisation of the Frequent Hypercyclicity Criterion in an associated separable weighted Banach space of harmonic functions which is densely embedded in $(\harmsp, d)$, but this general technique does not appear to be available in the case of the minimal growth rate \eqref{eq:minGro}. 
It is also worthwhile to note the qualitative difference  between  \eqref{eq:minGro} 
and the corresponding behaviour \eqref{hccase} of the  $\partial/\partial x_k$-hypercyclic harmonic functions from \cite{AA98a}.

Furthermore, for $N=2$ the claim  can be deduced from the entire function case, as given in~\cite{DS12}, by considering the real part of a corresponding frequently hypercyclic entire function possessing minimal $L^2$-growth (cf.~the comments which appear at the beginning of the proof of Proposition \ref{prop:FHCh}).  
Thus we are essentially concerned with the case $N \geq 3$, which  turns out   to involve different tools compared to the case $N = 2$.
The required harmonic function is constructed in Section \ref{sec:construction} and we prove it is frequently hypercyclic for $\partial/\partial x_k$ in Section \ref{sec:FHCh}.  
The argument is completed in  Section \ref{sec:GrowthRate}
by showing it has the desired minimal growth rate \eqref{eq:minGro}.

\section{Harmonic Polynomials} \label{sec:harmPolys}

We recall in this section the crucial background and auxiliary results from \cite{AA98a}  and \cite{BBGE10} regarding harmonic polynomials on $\mathbb{R}^N$ needed to prove Theorem \ref{thm:growthRate}.
The space of homogeneous harmonic polynomials  on $\mathbb{R}^N$ of 
homogeneity degree $m \geq 0$ is denoted by $\harmpolys{m}$.
The  harmonic analogue of the standard power series representation of 
holomorphic functions states that any $h \in \harmsp$ has a unique expansion of the form
\begin{equation}  \label{eq:DecompHomoHarmPolys}
h = \sum_{m=0}^\infty H_m
\end{equation}
where $H_m \in \harmpolys{m}$ for each $m \ge 0$ and the expansion converges in the metric $d$, see  \cite[Corollary 5.34]{ABR01}.  Moreover, $\langle H_j, H_k \rangle_r = 0$ when $j \neq k$, so  by orthogonality one has for any $r>0$ that
\begin{equation*}
M_2^2(h,r) = \sum_{m=0}^\infty M_2^2 (H_m, r).
\end{equation*}
The references \cite{ABR01} and \cite{AG01} contain further useful background information on harmonic functions and the spaces  $\harmpolys{m}$ are discussed in detail in   \cite[Chapter 5]{ABR01} and \cite[Chapter 2]{AG01}.

It will be enough to prove Theorem  \ref{thm:growthRate}  in the case of  $\partial/\partial x_1$ and this will be our standing assumption in the sequel. The cases $\partial/\partial x_k$ for 
$k = 2, \dotsc, N$ can be dealt with analogously.

For any   $x = (x_1, \dotsc, x_N) \in \mathbb{R}^N$, we recall  a function $f \colon \mathbb{R}^N \to \mathbb{R}$ is said to be \emph{$x_1$-axial} if $f(x)$ depends only on $x_1$ and $\left( x_2^2 + \cdots + x_N^2 \right)^{1/2}$.
This means $f$ is invariant under rotation around the $x_1$-axis, that is
\begin{equation*}
f(x_1, x_2, \dotsc, x_N) = f(y_1, y_2, \dotsc, y_N)
\end{equation*}
whenever $x_1 = y_1$ and $x_2^2 + \cdots + x_N^2 = y_2^2 + \cdots + y_N^2$.

Kuran~\cite{Kur71} used  $x_1$-axial polynomials to construct a specific orthogonal representation of  $\harmpolys{m}$, see \eqref{eq:uniqRepH} below, which will be crucial for our construction. 
The starting point is the following fact due to Brelot and Choquet~\cite[Proposition 4]{BC55}
(see also  \cite[2.3.8]{AG01}), where $d_{m,N}$ denotes the dimension of  $\harmpolys{m}$:  
there exists an $x_1$-axial polynomial $I_{m,N}$ in $\harmpolys{m}$ satisfying
 $I_{m,N}(1,0,\dotsc,0) = 1$, for which every $x_1$-axial element of $\harmpolys{m}$ is proportional to $I_{m,N}$ and
\begin{equation*}
M_2(I_{m,N},r) = r^m \left( d_{m,N} \right)^{-1/2}
\end{equation*}
 for $r > 0$.
It can be shown~\cite[Proposition 5.8]{ABR01} that $d_{0,2} = 1$ and
\begin{equation*} \label{dim:dmN}
d_{m,N} = \frac{N+2m-2}{N+m-2}\binom{N+m-2}{m}
\end{equation*}
for $N+m \geq 3$.

 Kuran~\cite{Kur71} defined homogeneous (but not necessarily harmonic) polynomials of degree $m>0$ on $\mathbb{R}^N$ by   
\begin{equation*}
I^*_{m,N+2p}(x_1, \dotsc, x_N) = I_{m,N+2p}(x_1, \dotsc, x_N,\overbrace{0,\dotsc,0}^{2p})
\end{equation*}
for $p \in \mathbb{N}$.  We denote by $\mathcal{H}_m^0(\mathbb{R}^N)$  the subspace
\begin{equation*}
\mathcal{H}_m^0(\mathbb{R}^N)  = \left\lbrace H \in \harmpolys{m} : \partial H / \partial x_1 \equiv 0 \right\rbrace .
\end{equation*}
We will need the following reformulation of \cite[Theorems 2,3]{Kur71} which is recalled from \cite[Lemma 3]{AA98a}.
\begin{prop}
Let  $m \in \mathbb{N}$.
\begin{enumerate}
\item For $p \in \mathbb{N}$, if $u \in \mathcal{H}_p^0(\mathbb{R}^N)$ then $uI^*_{m,N+2p} \in \harmpolys{m+p}$ and
\begin{equation*}
d_{m,N+2p} M_2^2(uI^*_{m,N+2p}, 1) = M_2^2(u, 1).
\end{equation*}

\item If $H \in \harmpolys{m}$ then $H$ has a unique representation 
\begin{equation} \label{eq:uniqRepH}
H = \sum_{p=0}^m u_p I^*_{m-p,N+2p}
\end{equation}
where $u_p \in \mathcal{H}_p^0(\mathbb{R}^N)$ for $p = 0, \ldots , m$ and the terms in \eqref{eq:uniqRepH} are mutually orthogonal with respect to $\langle \: \cdot \, , \, \cdot \: \rangle_r$.
\end{enumerate}
\end{prop}

The preceding result allowed 
Aldred and Armitage~\cite{AA98a} to define linear maps 
$P_k \colon \harmpolys{m} \to \harmpolys{m+k}$,  for $k \geq 0$, by
\begin{equation} \label{defn:Primitive}
P_k(H) = \sum_{p=0}^m \frac{(m-p)!}{(m-p+k)!} u_p I^*_{m-p+k,N+2p}
\end{equation}
where $H \in \harmpolys{m}$ has the representation \eqref{eq:uniqRepH}.
In view of the following fundamental lemma, taken from \cite[Lemma 4]{AA98a}, we will refer to $P_k(H)$ as the $k^\mathrm{th}$ \emph{primitive} of $H$.
\begin{lma} \label{lma:PrimitiveProps}
Let $m, k \geq 0$ and $N \geq 2$. If $H \in \harmpolys{m}$  then $P_k(H) \in \harmpolys{m+k}$,
\begin{equation*}
\frac{\partial^k}{\partial x_1^k} P_k(H) = H
\end{equation*}
and
\begin{equation}  \label{ineq:primUpperBnd}
M_2^2 (P_k(H), 1) \leq c_{k,m,N} M_2^2 (H,1)
\end{equation}
where 
\begin{equation*}
c_{k,m,N} = \frac{(N+2m-2)!}{k!(N+2m+k-3)!(N+2m+2k-2)}.
\end{equation*}
\end{lma}

For fixed $m$ we will use the simpler estimate 
\begin{equation}
 \left( c_{k,m,N} \right)^{1/2} \leq \frac{c_m}{(k + m)!(k + m +1)^{N/2 -1}} \label{ineq:GrowthEst}
\end{equation}
for $k \in \mathbb{N}$, (cf.~line (4.2) in  \cite[p.~52]{BBGE10}).  Here
\begin{equation}  \label{defn:c_m}
c_m = c(m,N)
\end{equation} 
and the exact bound in \eqref{defn:c_m} is not important for our purposes, but we may assume that $m \mapsto c_m$ is increasing. 

Finally, the following  compatibility property of the different maps $P_k$ 
defined by \eqref{defn:Primitive} will be technically convenient.

\begin{lma}\label{lm:compat}
Let $H \in \harmpolys{m}$ and $k, \ell \geq 0$.  Then
\begin{equation*}
P_{k+\ell}(H) = P_k \left( P_\ell(H) \right).
\end{equation*}
\end{lma}

\begin{proof}
$P_\ell(H) \in \harmpolys{m+\ell}$ and hence
\begin{align*}
P_k \left( P_\ell(H) \right) &= P_k \left( \sum_{p=0}^m \frac{(m-p)!}{(m-p+\ell)!} u_p I^*_{m-p+\ell,N+2p} \right) \\
&= \sum_{p=0}^m \frac{(m-p)!}{(m-p+k+\ell)!} u_p I^*_{m-p+k+\ell,N+2p} = P_{k+\ell}(H).
\end{align*}
\end{proof}

\section{Construction of the Harmonic Function $h$}  \label{sec:construction}

Let $N \geq 2$  be fixed. The set of harmonic polynomials on $\mathbb{R}^N$ is 
dense in  the separable Fr\'{e}chet space $(\harmsp, d)$, so we can fix a $d$-dense sequence of harmonic polynomials $ \left( F_k \right) \subset \harmsp$. 
For technical simplicity we will also assume that each polynomial is repeated infinitely often
in the sequence.
For each $k \geq 0$  we let $m_k$ be the degree of $F_k$ and by
 \eqref{eq:DecompHomoHarmPolys} there is a unique representation
\begin{equation} \label{decompF_k}
F_k = \sum_{j=0}^{m_k} H_{k,j}
\end{equation}
where $H_{k,j} \in \harmpolys{j}$  for $j = 0, \dotsc, m_k$.   
 We associate with each $F_k$ an odd integer $\ell_k \in \mathbb{N}$, 
 so that the sequence $(\ell_k)$ is strictly increasing and
 \begin{equation} \label{defn:l_k}
 \ell_k \geq   2m_k + c_{m_k} 2^k \left( M_2^2 \left( F_k, 1 \right) + 1 \right)
 \end{equation}
 where $c_{m_k}$ is as defined in \eqref{defn:c_m}.
The final choice for $(\ell_k)$ will  be made later in \eqref{defn:choice}, depending on a 
given constant $C > 0$.

We can  unambiguously define the $n^\mathrm{th}$ primitive of $F_k$  for all $n \in \mathbb{N}$ as
\begin{equation*}
P_n(F_k) = \sum_{j=0}^{m_k} P_n(H_{k,j})
\end{equation*}
where each $P_n(H_{k,j}) \in \harmpolys{n+j}$ is as defined in \eqref{defn:Primitive}. It follows 
from Lemma \ref{lma:PrimitiveProps}  that
\begin{equation*}
\frac{\partial^n}{\partial x_1^n} P_n(F_k) = F_k.
\end{equation*}
We next introduce the sets 
\begin{equation}\label{def:arith}
\mathcal{A}_k = \left\lbrace (2m-1) \ell_k 2^k : m \geq 1 \right\rbrace
\end{equation}
for each $k \geq 1$. Note that the $\mathcal{A}_k$ are pairwise disjoint 
infinite arithmetic sequences, so that
$ \bigcup_{k\geq 1} \mathcal{A}_k$ is a partition of some subset of the even natural numbers.

To construct the required harmonic function, we define for each $n \in  \mathcal{A}_k$
a harmonic polynomial $Q_n$ which is a finite sum of suitable primitives of  the  
harmonic polynomial $F_k$. 

We first let   $Q_n = 0$
whenever $n \notin  \bigcup_{k\geq 1} \mathcal{A}_k$ 
(this includes all odd integers $n$) or $n = 0$.
Suppose next that the even integer $n \in \mathcal{A}_k$, for a fixed unique $k$.  
If  $n < 10\ell_k$, we set $Q_n = 0$ and for $n = (2m-1)\ell_k 2^k \geq 10\ell_k$  we define
\begin{equation*}
Q_n = \sum_{j = 1}^{(2m-1)2^k} P_{n^2 + j\ell_k} (F_k). 
\end{equation*}
Finally, we define 
\begin{equation} \label{defn:h}
h = \sum_{n=1}^\infty  Q_n = \sum_{k=1}^\infty \sum_{n \in \mathcal{A}_k} Q_n
\end{equation}
and we proceed to show in Sections \ref{sec:FHCh} and \ref{sec:GrowthRate} 
that $h$ satisfies the claims of Theorem \ref{thm:growthRate}. 
Note  we must also verify that $h$ is defined on the whole of 
$\mathbb{R}^N$ and that $h \in \harmsp$.  Since the details are closely related to the estimates in Section \ref{sec:FHCh}
 we defer this discussion until Remark \ref{rm:harmonic}.

The estimates in Sections \ref{sec:FHCh} and \ref{sec:GrowthRate} 
 will frequently use the fact that the respective sets of homogeneity degrees of 
the harmonic polynomials appearing in $Q_n$ and $Q_{n'}$ are disjoint whenever 
$n \neq n'$.  That is  $\langle \, Q_n, \, Q_{n'} \, \rangle_r = 0$ for any $r > 0$, so that 
 $Q_n$ and $Q_{n'}$ are orthogonal for $M_2(\: \cdot\: , \: r)$. In fact, if 
 $F_k = \sum_{q=0}^{m_k} H_{k,q}$ and $n = (2m-1)\ell_k 2^k \in \mathcal{A}_k$, then the homogeneity degrees related to  $Q_n =  \sum_{j = 1}^{(2m-1)2^k} P_{n^2 + j\ell_k} (F_k)$ are contained in the intervals $[n^2,  n^2 + (2m-1)\ell_k 2^k + m_k]$. These intervals are pairwise 
 disjoint for different $n$, since
\begin{equation*}
 n^2 + (2m-1)\ell_k 2^k + m_k = n^2 + n + m_k < (n+1)^2
\end{equation*}
 in view of  \eqref{defn:l_k}.

\section{Frequent Hypercyclicity of $h$}  \label{sec:FHCh}

The aim of this section is to prove  in Proposition \ref{prop:FHCh} that the function $h$ defined in \eqref{defn:h} is a frequently hypercyclic vector in $\harmsp$ for the partial differentiation operator
$\partial /\partial x_1$.  Towards this end, we first point out that 
convergence in the average $L^2$-norm on the  
sphere $S(2r)$ of radius $2r$ gives convergence in the sup-norm on $S(r)$
for any $r>0$. This depends on basic facts about the Poisson kernel, which we first recall (the details can be found in \cite[Chapter 1]{ABR01} or \cite[Section 1.3]{AG01}). 

We  denote the open ball of radius $r>0$ centred at the origin of $\mathbb{R}^N$ by $B(r)$ and  
put $B = B(1)$.  Moreover, let   
$\overline{B}(r)$ be the closed  ball and $S = S(1)$ the unit sphere.  
The \emph{Poisson kernel}  of $B$ is the function $P \colon B \times S \to \mathbb{R}$ defined as 
\begin{equation*}
P(x,y) = \frac{1 - \lvert x \rvert^2}{\lvert x - y \rvert^N}
\end{equation*}
for $x \in  B$ and $y \in S$. 
It  is well known  for  any harmonic function $h \in \mathcal{H}(\overline{B})$ that we have 
\begin{equation}  \label{poissonRep}
 h(x)  = \int_S  P(x,y) h(y)  \; \mathrm{d}\sigma(y)
\end{equation}
for all $x \in B$ and where  $\sigma = \sigma_1$ is the normalised $(N-1)$-dimensional measure on $S$, as introduced in Section \ref{sec:background}.

\begin{lma}  \label{lma:convergenceLUC}
Let $h$ be a harmonic function on $\overline{B}(2r) \subset \mathbb{R}^N$, for $r>0$ 
and $N \geq 2$.  Then 
\begin{align*}
\sup_{\lvert x \rvert = r} \lvert h(x) \rvert \leq c_N M_2(h, 2r)
\end{align*}
where $c_N$ is a constant depending only on $N$.
\end{lma}

\begin{proof}
For any $x \in B$ with $\lvert x \rvert \leq 1/2$, it follows from \eqref{poissonRep} and the Cauchy-Schwarz inequality that
\begin{align}
\lvert h(x) \rvert &\leq \int_S P(x,y) \lvert h(y) \rvert \; \mathrm{d}\sigma(y) \nonumber \\
&\leq \left( \int_S \left( P(x,y) \right)^2 \, \mathrm{d}\sigma(y) \right)^{1/2} M_2(h,1) = c_N M_2(h,1) \label{ineq:poisson}
\end{align}
where $c_N$  depends only on $N$.

Since dilations preserve harmonicity, we can extend \eqref{ineq:poisson} to any ball of radius $2r >0$.  To see this let
 $h_{2r}(x) = h(2rx)$ and notice for any $x \in B$ with $\lvert x \rvert \leq r$ that according to the appropriate normalisations  
\begin{align*}
\sup_{\lvert x \rvert \leq 1/2} \lvert h_{2r}(x) \rvert &\leq c_N M_2(h_{2r}, 1) \\
&= c_N \left( \int_S \lvert h(2ry) \rvert^2 \; \mathrm{d}\sigma(y) \right)^{1/2}\\
&= c_N \left( \int_{S(2r)} \lvert h(y) \rvert^2 \; \mathrm{d}\sigma_{2r}(y) \right)^{1/2}  = c_N M_2(h,2r).
\end{align*}
This yields the claim.
\end{proof}

\begin{prop} \label{prop:FHCh}
Let $h \in \harmsp$ be as defined in \eqref{defn:h}.  Then $h$ is a frequently hypercyclic vector 
in $\harmsp$ for the partial differentiation operator $\partial /\partial x_1$ for any strictly increasing sequence $(\ell_k)$ satisfying \eqref{defn:l_k}.
\end{prop}

\begin{proof}
We first note that for $N=2$ the complete result in Theorem \ref{thm:growthRate} can be deduced from the corresponding case for the entire functions in \cite{DS12}. In fact, let $f_0 = u_0 + iv_0$ be a frequently hypercyclic entire function for the differentiation operator $g \mapsto  g'$
having minimal $L^2$-growth on $S(r)$. Then $u_0 = \mathrm{Re}(f_0)$ is a harmonic function on $\mathbb{C} = 
\mathbb{R}^2$ and $M_2(u_0, r) \le M_2(f_0, r)$ for all $r > 0$.
It is not difficult to check that $u_0$ is a
frequently hypercyclic vector for $\partial /\partial x_1$ in ${\mathcal{H}(\mathbb{R}^2)}$.
Consequently we may (and will) assume for the rest of the argument that $N \geq 3$.

To begin the actual argument, for any $k \ge 1$ and $n = (2m-1)\ell_k 2^k \in \mathcal{A}_k$  
with $n \geq 10\ell_k$  let
\begin{equation*}
\mathcal{B}_{n,k} = \left\lbrace  n^2 + j\ell_k  \, : \,   1 \leq j \leq (2m-1)2^k \right\rbrace.
\end{equation*}
We claim that the union
\begin{equation*}
\mathcal{B}_k = \bigcup_{n \in \mathcal{A}_k} \mathcal{B}_{n,k}
\end{equation*}
has positive lower density for any  $k \geq 1$.
In fact, suppose $n = (2m-1)\ell_k 2^k \in \mathcal{A}_k$ 
for some $m \ge 1$ and consider a given  integer $t \in [n^2,\, n^2 + n]$.
By inspection $\mathcal{B}_k$ contains $(2u - 1)2^k - 1$ integers
from the interval 
\begin{equation*}
\left[ (2u-1)^2\ell_k^2 2^{2k},\, (2u+1)^2 \ell_k^2 2^{2k} \right)
\end{equation*}
 for each $u = 1, \dotsc, m-1$, so that
\begin{align*}
\frac{\# \left( \mathcal{B}_{k} \cap \left\lbrace s \in \mathbb{N} : 0 \leq s \leq t \right\rbrace \right)}{t} & \geq
\frac{ 2^k \left(  \sum_{u=1}^{m-1} (2u-1)\right)  - (m-1)}{(2m+1)^{2} \, \ell_k^2 \, 2^{2k}} \\
& = \frac{(m-1)^{2}}{(2m+1)^{2} \, \ell_k^2 \, 2^k} - \frac{(m-1)}{(2m+1)^{2} \, \ell_k^2 \, 2^{2k}}
\end{align*}
which tends to $\frac{1}{4 \ell_k^2 \, 2^{k}}$ as $m \to \infty$. 
Clearly this estimate yields that $\mathcal{B}_k$ 
has positive lower density.

Fix $r>0$ and let $\widetilde{F} \in (F_u)_{u\geq1}$ as well as  $\varepsilon > 0$ be given.  
By construction we can find an $F_k$ from our dense sequence $(F_k)$ such that 
$\widetilde{F} = F_k$ and the corresponding integer $\ell_k \geq  (er)^2$.
We  claim  for $k$ large enough it holds that
\begin{equation}   \label{claimFHC}
M_2\left( \widetilde{F} - \frac{\partial^s}{\partial x_1^s} h, r \right) \leq \varepsilon
\end{equation}
for  all $s \in \mathcal{B}_{n,k}$, where   $n = (2m-1)\ell_k 2^k \in \mathcal{A}_k$ and 
$n \geq 10 \ell_k$.

Suppose  $s =  n^2 + t\ell_k \in \mathcal{B}_{n,k}$ for some integer $1 \leq t \leq (2m-1)2^k$.
To compute $\frac{\partial^s}{\partial x_1^s} h$ note first that 
$\frac{\partial^s}{\partial x_1^s} Q_{n'} = 0$ for $n ' < n$. Moreover, by construction  and
Lemma \ref{lm:compat} we get after relabelling that 
\begin{align*}
\frac{\partial^s}{\partial x_1^s} Q_n 
& = \sum_{j=1}^{(2m-1)2^k} \frac{\partial^s}{\partial x_1^s} P_{n^2 + j\ell_{k}}(F_k)
= F_k + \sum_{j=t+1}^{(2m-1)2^k} \frac{\partial^s}{\partial x_1^s} P_{n^2 + j\ell_k}(F_k)\\
& = F_k + \sum_{j=1}^{(2m-1)2^k-t} P_{j \ell_k}(F_k).
\end{align*}
Here we used that  $\frac{\partial^s}{\partial x_1^s} P_{n^2 + j\ell_{k}}(F_k) = 0$
for $j < t$  since  $n^2 + j\ell_k + m_k < s$.
The next term in $h$ is $Q_{n+1} \equiv 0$, since $n+1$ is odd and hence
\begin{equation}  \label{eq:sumsOfh}
 \frac{\partial^s}{\partial x_1^s} h  - F_k = \sum_{j=1}^{(2m-1)2^k-t} P_{j\ell_k} \left( F_k \right) + \sum_{j=n+2}^\infty \frac{\partial^s}{\partial x_1^s} Q_j  \eqqcolon F+G.
\end{equation}

Before we proceed to estimate $M_2^2 \left( \frac{\partial^s}{\partial x_1^s} h  - F_k, r \right)$,  it is   convenient to calculate first an upper bound that will be needed later in this proof.  For any $n, k \in \mathbb{N}$ we consider the primitive $P_n (F_k)$, where $F_k$ is a harmonic polynomial from our dense sequence $(F_k)$.
By orthogonality and \eqref{decompF_k} we have for any $r>0$ that
\begin{equation*}
M_2^2  \left( P_n (F_k), r \right) \leq \sum_{i=0}^{m_k} M_2^2 \left( P_n (H_{k,i}), r \right) 
\end{equation*}
and furthermore by  \eqref{ineq:primUpperBnd}, \eqref{ineq:GrowthEst} and homogeneity it follows for each $i= 0,\dotsc, m_k$ that
\begin{align*}
M_2^2 \left( P_n (H_{k,i}), r \right)  \leq \frac{c_i r^{2(n+i)} M_2^2 \left( H_{k,i}, 1 \right)}{(n+i)!^2 (n+i+1)^{N-2}}.
\end{align*}
Applying Stirling's formula and the fact that $c_i \leq c_{m_k}$  for  $i = 0, \dotsc, m_k$, it follows that
\begin{equation*}
\frac{c_i r^{2(n+i)}}{(n+i)!^2 (n+i+1)^{N-2}} \leq \frac{c_{m_k} (er)^{2(n+i)}}{2\pi (n+i)^{2(n+i)} (n+i)^{N-1}}.
\end{equation*}
Combining the above and using orthogonality we get that
\begin{equation}  \label{bound:PrimF_k}
M_2^2  \left( P_n (F_k), r \right) \leq \frac{c_{m_k}}{2\pi n^{N-1}} \sum_{i=0}^{m_k} \frac{ (er)^{2(n+i)} M_2^2 \left( H_{k,i}, 1 \right)}{(n+i)^{2(n+i)}}.
\end{equation}

Following these preparations we next  estimate $M_2^2 \left(F,r \right)$, where $F$ is as defined on the right hand side of \eqref{eq:sumsOfh}.  For each harmonic polynomial $P_{j\ell_k} \left( F_k \right)$ we get from  \eqref{bound:PrimF_k} that
\begin{equation*}
M_2^2 \left( P_{j\ell_k} \left( F_k \right), r \right) \leq \frac{c_{m_k}}{2\pi \ell_k^{N-1}} \sum_{i=0}^{m_k} \frac{(er)^{2(j \ell_k + i)} M_2^2 \left( H_{k,i}, 1  \right)}{(j\ell_k + i)^{2(j\ell_k + i)}}.
\end{equation*}
By the fact that $\ell_k \geq  (er)^2$ it follows that 
\begin{equation*}
\frac{(er)^{2(j \ell_k + i)}}{(j\ell_k + i)^{(j\ell_k + i)}}  \leq 1
\end{equation*}
and hence by orthogonality 
\begin{equation*}
M_2^2 \left( P_{j\ell_k} \left( F_k \right), r \right) \leq \frac{c_{m_k}}{2\pi  \ell_k^{N-1} \ell_k^j} M_2^2( F_k, 1).
\end{equation*}
By summing up we have that
\begin{align*}
 \sum_{j=1}^{(2m-1)2^k-t}  M_2^2 \left( P_{j\ell_k} \left( F_k \right), r \right) &\leq \frac{c_{m_k}}{2\pi \ell_k^{N-1}} M_2^2( F_k, 1) \sum_{j=1}^\infty  \frac{1}{\ell_k^j}\\
 &\leq \frac{1}{2\pi \ell_k^{N-1}} \cdot \frac{c_{m_k} M_2^2( F_k, 1) }{\ell_k} \left( \frac{1}{1 - \frac{1}{\ell_k}} \right) \\
 &\leq \frac{1}{2 \ell_k^{N-1}} 
\end{align*} 
where we applied \eqref{defn:l_k} and a geometric series estimate. We conclude that
\begin{equation}\label{estimate01}
 M_2^2 \left( F, r\right) \le \frac{1}{2 \ell_k^{N-1}}.
\end{equation}

Next we estimate $M_2^2 \left(G,r \right)$, where $G$ 
is defined on the right hand side of \eqref{eq:sumsOfh}.
Suppose $j \geq n+2$ with $j = (2m'-1)\ell_{k'} 2^{k'} \in \mathcal{A}_{k'}$ for some 
$k', m' \ge 1$. 
By orthogonality we have 
\begin{equation*}
M_2^2 \left(\frac{\partial^s}{\partial x_1^s} Q_j, r \right) = 
\sum_{q=1}^{(2m'-1)2^{k'}} M_2^2 \left( P_{j^2 -s + q \ell_{k'}}(F_{k'}), r \right)
\end{equation*}
and  by  \eqref{bound:PrimF_k} it follows for each harmonic polynomial $P_{j^2 -s + q \ell_{k'}}(F_{k'})$ that
\begin{align*}
M_2^2 & \left( P_{j^2 -s + q \ell_{k'}}(F_{k'}), r \right) \\
&\leq \frac{c_{m_{k'}}}{2\pi (j^2 -s + q \ell_{k'})^{N-1}} \sum_{i=0}^{m_{k'}} \frac{(er)^{2(j^2 -s + q \ell_{k'}+i)} M_2^2 \left( H_{k',i}, 1 \right)}{(j^2 -s + q \ell_{k'} +i)^{2(j^2 -s + q \ell_{k'} +i)}}.
\end{align*}
Recall next that $1 \leq t \leq (2m-1)2^k$ and $n \geq t\ell_k$,   so that
\begin{equation}  \label{ineq:j-s}
j^2 -s  \geq (n+2)^2   - n^2  - t \ell_k  > 2n.
\end{equation}
Applying the fact that $n \geq 10 \ell_k \geq  (er)^2$ one also has 
\begin{equation*}
\frac{ (er)^{2(j^2 -s + q \ell_{k'} +i)} }{(2n + q \ell_{k'} )^{j^2 -s + q \ell_{k'} +i} } \leq 1
\end{equation*}
for $i = 0,\dotsc, m_{k'}$. By adding these estimates, applying \eqref{ineq:j-s} and using the fact that $n \geq 10 \ell_k$ we get  again by orthogonality  that
\begin{equation*}
 M_2^2  \left( P_{j^2 -s + q \ell_{k'}}(F_{k'}), r \right) 
 \le   \frac{c_{m_{k'}} M_2^2 \left( F_{k'}, 1 \right)}{2\pi \ell_k^{N-1} (2n + q \ell_{k'})^{j^2 -s + q \ell_{k'}} }.
\end{equation*}
From this estimate we get that 
\begin{align*} 
M_2^2 \left(\frac{\partial^s}{\partial x_1^s} Q_j, r \right)  &= \sum_{q=1}^{(2m'-1)2^{k'}}  M_2^2 \left( P_{j^2 -s + q \ell_{k'}}(F_{k'}), r \right) \\
&\leq \sum_{q=1}^\infty \frac{c_{m_{k'}} M_2^2 \left( F_{k'}, 1 \right)}{2\pi \ell_k^{N-1} (2n + q \ell_{k'})^{j^2 -s + q \ell_{k'}} }  \\
&\leq \frac{c_{m_{k'}} M_2^2 \left( F_{k'}, 1 \right)}{2\pi \ell_k^{N-1} (2n)^{j^2 -s} }   \sum_{q=1}^\infty \frac{1}{\ell_{k'}^q} \\
&\leq \frac{1}{2\pi \ell_k^{N-1} (2n)^{j^2 -s} }  \cdot \frac{c_{m_{k'}} M_2^2 \left( F_{k'}, 1 \right)}{\ell_{k'}} \left( \frac{1}{1 - \frac{1}{\ell_{k'}}} \right) \\
&\leq \frac{1}{2 \ell_k^{N-1} (2n)^{j^2 -s} }
\end{align*}
where we again applied \eqref{defn:l_k} and a crude geometric series estimation.

By summing over  $j \geq n+2$  we arrive at  
\begin{align}
\sum_{j=n+2}^\infty M_2^2 \left(\frac{\partial^s}{\partial x_1^s} Q_j, r \right) \leq \frac{ 1 }{2 \ell_k^{N-1}} \sum_{j=n+2}^\infty  \frac{ 1 }{ (2n)^{j^2 -s} } \leq \frac{ 1 }{2 \ell_k^{N-1} }   \label{estimate02}
\end{align}
where we again used  a geometric series estimation  and that $n \geq 10 \ell_k$.

By combining \eqref{estimate01} and \eqref{estimate02}, and taking into account 
orthogonality and the fact that $N \geq 3$, we finally get  that
\begin{equation*}
M_2^2\left( \widetilde{F} - \frac{\partial^s}{\partial x_1^s} h, r \right) \leq  \frac{1}{\ell_k^2}.
\end{equation*}
Since $\widetilde{F} = F_k$ for infinitely many $k$ and the sequence $(\ell_k)$ is 
strictly increasing, we can certainly find $k$ such that $\ell_k^{-2} < \varepsilon$. 

In conclusion, we have shown that we can estimate $\widetilde{F}$ up to any given 
$\varepsilon > 0$ by partial derivatives  of $h$ 
associated with $\mathcal{A}_k$ in the $L^2$-norm 
on the sphere $S(r)$ for any fixed $r > 0$. By applying
Lemma \ref{lma:convergenceLUC} we obtain a similar estimate 
in the sup-norm on the closed ball $\overline{B}(\frac{r}{2})$. This completes
the proof of Proposition \ref{prop:FHCh}, since for each $k$ the set $\mathcal{B}_k$ 
corresponding to  $\mathcal{A}_k$ has positive lower density.
\end{proof}

This is  a suitable point to verify that  $h$ defined in \eqref{defn:h} does indeed define a harmonic function 
on $\mathbb{R}^N$. 

\begin{remark}\label{rm:harmonic}
We claim that $h \in \harmsp$. Let $r > 0$ be fixed. 
Observe that if $n \in \mathcal{A}_k$ and we take $s = 0$ 
in the growth estimates for the 
remainder $G_0 = \sum_{j=n+2}^\infty Q_j$ in  \eqref{eq:sumsOfh}, then by following the 
argument from Proposition \ref{prop:FHCh} we obtain that
\begin{equation*}
M_2^2\left( \sum_{j=n+2}^\infty Q_j, r \right) \leq  \frac{1}{\ell_k^2}
\end{equation*}
for all large enough $k$ (depending on $r$). 
This implies by Lemma \ref{lma:convergenceLUC} that the remainder term  of the series 
defining $h$ in \eqref{defn:h} converges uniformly to $0$ on the closed ball $\overline{B}(r/2)$ for any fixed $r > 0$. 
By completeness the partial sums of $h$
then converge to a harmonic function defined on the whole of $\mathbb{R}^N$. 
\end{remark}

\section{Growth Rate of $h$}  \label{sec:GrowthRate}

In this section we complete the proof of Theorem \ref{thm:growthRate}
by showing that the frequently hypercyclic harmonic function $h$ from \eqref{defn:h}
has the desired minimal $L^2$-growth rate as soon as the sequence $(\ell_k)$ from
\eqref{defn:l_k} grows fast enough.
For this purpose we need the following useful  lemma which is a variant of \cite[lemma 2.2]{BBGE10}. 
\begin{lma}  \label{lma:growthRate} 
Let  $N\geq 2$ be given. Then there exists a constant $C>0$ such that for 
all given integers $\ell\geq 1$, $u\in \{0,\dotsc, \ell-1\}$ and radii $r>0$ it holds that
\begin{equation*}
\sum_{k=\ell}^\infty \frac{r^{2(\ell k+u)}}{(\ell k +u)!^2(\ell k+u+1)^{N-2}} \leq \frac{C}{\ell} \cdot \frac{e^{2r}}{r^{N-3/2}}.
\end{equation*}
\end{lma}

For the proof of the lemma we record the following useful observation.

\begin{lma}\label{lma:jokatoinen} 
Let $(a_n)_{n\geq 0}$ be a summable sequence of non-negative real numbers, for which there exists $n_0\geq 2$ such  that the elements $n \mapsto a_n$ are increasing for $n\leq n_0$
and decreasing for $n\geq n_0$. Then for any $\ell\geq 2$ and 
any $u\in\{ 0,\dotsc, \ell-1\}$ we have 
\begin{equation*}
\sum_{k=1}^\infty a_{k\ell+u}\;\leq\; \ell^{-1} \left(\sum_{n=0}^\infty a_n\right) +2\sup_{n\geq 0}a_n.
\end{equation*}
In the case when the sequence $(a_n)_{n\geq 0}$ is decreasing we have the stronger estimate
\begin{equation*}
\sum_{k=1}^\infty a_{k\ell+u} \;\leq\; \ell^{-1}\left( \sum_{n=0}^\infty a_n \right). 
\end{equation*}
\end{lma}

\begin{proof} Let  $k\geq 0$ be given. If $k\ell+u\leq n_0-\ell+1$ we obtain that
\begin{equation*}
a_{k\ell+u}\leq \ell^{-1}(a_{k\ell+u}+a_{k\ell+u+1} + \cdots +a_{k\ell+u+\ell-1})
\end{equation*}
since the sequence $(a_n)$ is increasing for these indices.  
Correspondingly, if $k\ell+u\geq n_0+\ell$ we get that 
\begin{equation*}
a_{k\ell+u}\leq \ell^{-1}(a_{k\ell+u-\ell+1}+a_{k\ell+u-\ell+2} + \cdots +a_{k\ell+u})
\end{equation*} 
since $(a_n)$ is decreasing here. 
We obtain the claim by summing these estimates over all possible values of $k$ and noting that
all indices $k$ such that $k\ell+u\not\in [n_0-\ell+1,n_0+\ell]$ are covered,  apart from at
most two.
\end{proof}

\begin{proof}[Proof of Lemma \ref{lma:growthRate}]
We may assume without loss of generality that $r \ge 2$. 
We start by recalling the following estimate, which can be found in \cite[Lemma 2.2]{BBGE10}
\begin{equation} \label{eq:help1b}
\sum_{k=0}^\infty \frac{r^{2n}}{n!^2} \lesssim \frac{e^{2r}}{r^{1/2}}, \quad \textrm{for } r>0.
\end{equation}
Here (and below) $\lesssim$ denotes an inequality up to a numerical constant. 
By noting that 
\begin{equation}
\frac{r^{2(n+1)}/(n+1)!^2}{r^{2n}/n!^2}= \left( \frac{r}{n+1} \right)^2
\end{equation}
and using Stirling's formula  to estimate both the $\lfloor r\rfloor^\mathrm{th}$ and the $\left( \lfloor r\rfloor +1 \right)^\mathrm{th}$ terms,  we  easily verify that the maximal term 
of the above series satisfies the estimate
\begin{equation}\label{eq:help1c}
\frac{r^{2n}}{n!^2}\lesssim \frac{e^{2r}}{r}.
\end{equation}
From \eqref{eq:help1b} we deduce immediately that
\begin{equation} \label{eq:help1}
\sum_{n\geq \floor{r/2}} \frac{r^{2n}}{n!^2(n+1)^{N-2}} \lesssim \frac{e^{2r}}{r^{N-3/2}}, \quad \textrm{for } r>0.
\end{equation}
On the other hand, by the monotonicity of the terms and Stirling's formula we have 
the crude bound
\begin{align}  \label{eq:apu1}
\sum_{n< \floor{r/2}} \frac{r^{2n}}{n!^2(n+1)^{N-2}} &\leq \sum_{n<\floor{r/2}}\frac{r^{2n}}{n!^2} \leq \floor{r/2} \left(\frac{\floor{r/2}^{\floor{r/2}}}{\floor{r/2}!}\right)^2 \nonumber\\
&\lesssim r \left( \frac{e^{r/2}}{r^{1/2}} \right)^2 = e^r \lesssim e^{2r} r^{3/2 -N}.
\end{align}
In combination with \eqref{eq:help1} this yields that
\begin{equation} \label{eq:help2}
\sum_{n=0}^\infty \frac{r^{2n}}{n!^2(n+1)^{N-2}}\lesssim\frac{e^{2r}}{r^{N-3/2}}, \quad \textrm{for } r>0.
\end{equation}
We next observe that \eqref{eq:apu1} implies that the maximal term among the  first $\floor{r/2}$ 
terms  of the series
\begin{equation*}
\sum_{n=0}^\infty \frac{r^{2n}}{n!^2(n+1)^{N-2}}
\end{equation*} 
is dominated by $e^r$. On the other hand, in view of \eqref{eq:help1c} the remaining terms of this series  have the upper bound 
$\lesssim e^{2r}r^{1-N}$.   
A fortiori, the latter bound is an upper bound for all the terms of the sum \eqref{eq:help2}. 

Next observe that for each $r > 0$ 
the function 
\begin{equation*}
n\mapsto \log\left(\frac{r^{2n}}{n!^2(n+1)^{N-2}}\right)
\end{equation*} 
is  concave for $n\geq n_0(N)$, where the bound only depends on $N$.
Assuming this for a moment, we complete the proof of Lemma \ref{lma:growthRate} as follows.
The concavity allows us to invoke Lemma  \ref{lma:jokatoinen} for large enough $\ell\geq \ell_0(N)$ and we deduce that
\begin{equation*}
\sum_{k=2\ell}^\infty \frac{r^{2(\ell k+u)}}{(\ell k +u)!^2(\ell k+u+1)^{N-2}}\lesssim  
\frac{C}{\ell} \cdot \frac{e^{2r}}{r^{N-3/2}} + e^{2r}r^{1-N}.
\end{equation*}
For $r \geq \ell^2$ this inequality immediately yields the desired estimate. 
On the other hand, if $r< \ell^2$ one notes that the left hand series is  decreasing starting from the index $k=\ell$, 
whence the claim follows directly from the second statement in Lemma \ref {lma:jokatoinen}.

To verify the claim about concavity we consider the related function
\begin{equation*}
\psi(x) \coloneqq 2x\log r -\left(  \log\left( \Gamma(x+1) \right) + (N-2)\log(x+1) \right).
\end{equation*}
It follows from (31) of \cite[p.~200]{Ahl78} that the second derivative of the logarithmic gamma 
function is
\begin{equation*}
\psi''(x) = - \sum_{n=0}^\infty \frac{1}{(x+n+1)^2} + \frac{N-2}{(x+1)^2}.
\end{equation*}
This implies that $\psi''(x) < 0$ for all $x > x_N$, since 
\begin{equation*}
\sum_{n=0}^{\floor{x}}  \frac{1}{(x+1+n)^2}  \ge \frac{x}{(2x+1)^2}.
\end{equation*}
\end{proof}

The  proof of Theorem \ref{thm:growthRate} is completed by the following proposition.

\begin{prop}\label{prop:growthRate}
Let $h \in \harmsp$ be as defined in \eqref{defn:h}.  Then 
for any given constant $C > 0$ there exists a choice of $(\ell_k)$ so that the 
$\partial/\partial x_1$-frequently hypercyclic harmonic function $h \in \harmsp$ satisfies
\begin{equation*}
M_2(h,r) \leq C \frac{e^r}{r^{N/2 - 3/4}}
\end{equation*}
for all  $r > 0$.
\end{prop}

\begin{proof}
Recall that $h$  has the representation
$h = \sum_{k=1}^\infty \sum_{n \in \mathcal{A}_k} Q_n$
whence
\begin{equation*}
M_2^2 \left( h, r \right) = \sum_{k=1}^\infty \sum_{n \in \mathcal{A}_k} M_2^2 \left(Q_n, r \right)
\end{equation*}
by orthogonality  for any $r>0$.
Moreover, for any fixed $k$ and any $n = (2m-1)\ell_k 2^k \in \mathcal{A}_k$ we  further obtain, by orthogonality, \eqref{decompF_k}, \eqref{ineq:primUpperBnd}, \eqref{ineq:GrowthEst},  as well as changing the order of summation, that   
\begin{align*}
 M_2^2 & \left( Q_n, r \right) =  \sum_{j=1}^{(2m-1)2^k}  M_2^2 \left( P_{n^2 + j\ell_k} (F_k), r \right) \\
&=  \sum_{j=1}^{(2m-1)2^k}  \sum_{q=0}^{m_k} M_2^2 \left( P_{n^2 + j\ell_k} \left(  H_{k,q} \right), r \right)  \\
& \leq  c_{m_k} \sum_{q=0}^{m_k} M_2^2 \left(  H_{k,q} , 1 \right) \sum_{j=1}^{(2m-1)2^k}  \frac{r^{2(n^2 + j\ell_k + q)} }{(n^2 + j\ell_k + q)!^2 (n^2 + j\ell_k + q+1)^{N-2}}.
\end{align*}
We also used above that $c_q \leq c_{m_k}$ for  $0 \leq q \leq m_k$.

Recall next that the sets $\left\lbrace n^2 + j \ell_k : 1 \le j \le (2m-1)2^k \right\rbrace$ are pairwise disjoint as the $n = (2m-1)\ell_k 2^k \in \mathcal{A}_k$ vary.
 Hence we may add the above estimates over 
the disjoint blocks of indices corresponding to 
$n \in \mathcal{A}_k$. 
Using the facts that $n \ge 10 \ell_k$, $\ell_k$ divides $n$, by orthogonality and applying Lemma  \ref{lma:growthRate} we obtain that
\begin{align}
\sum_{n \in \mathcal{A}_k} M_2^2  \left( Q_n, r \right)  & \le   c_{m_k}   \sum_{q=0}^{m_k} 
M_2^2 \left(  H_{k,q} , 1 \right) 
\sum_{j=2\ell_k}^{\infty}  \frac{r^{2(j\ell_k + q)} }{(j\ell_k + q)!^2 (j\ell_k + q+1)^{N-2}}  \nonumber \\
 & \le  c_{m_k} \sum_{q=0}^{m_k} \frac{C'M_2^2 \left(  H_{k,q} , 1 \right)}{\ell_k} \cdot 
 \frac{e^{2r}}{r^{N-3/2}}  \nonumber \\  
 & = C' \frac{c_{m_k} M_2^2  \left(  F_k, 1 \right)}{\ell_k} \cdot \frac{e^{2r}}{r^{N-3/2}}  \label{estimateQn}
\end{align}
 where the numerical constant $C' > 0$ from Lemma \ref{lma:growthRate}  is independent of   $\ell_k$ and $q = 0, \dotsc, m_k$.  

As the final step we successively choose the integers of 
the increasing sequence $(\ell_k)$   large enough, so that \eqref{defn:l_k} 
is satisfied and in addition
\begin{equation}\label{defn:choice}
C' \, \sum_{k=1}^\infty  \frac{c_{m_k} M_2^2  \left(  F_k, 1 \right) }{\ell_k} \le C^2
\end{equation}
where $C > 0$ is the given  constant in the proposition. 
By summing \eqref{estimateQn} over $k$ and taking \eqref{defn:choice} into account we arrive at the desired growth estimate
\begin{equation*}
M_2^2 \left( h, r \right) = \sum_{k=1}^\infty \sum_{n \in \mathcal{A}_k} M_2^2 \left(Q_n, r \right)
\le C^2  \frac{e^{2r}}{r^{N-3/2}}.
\end{equation*} 
\end{proof}

Combining  Propositions \ref{prop:FHCh} and \ref{prop:growthRate} we have established Theorem \ref{thm:growthRate}.

\subsection*{Concluding remarks} 

A natural further question is, what are the precise growth rates 
of the $L^p$-norm on $S(r)$ for $\partial/\partial x_k$-frequently hypercyclic 
harmonic functions $h \in \harmsp$ for $p \neq 2$?  A comparison with the  entire functions case~\cite{DS12} indicates that additional tools will be required. Note that
\cite{AA98b} and \cite{BBGE10} also contain results related to  
general partial differentiation operators $D^\alpha$ on  $\harmsp$ for 
$\alpha = (\alpha_1,\ldots,\alpha_N)$.

This paper forms part of the PhD thesis of Clifford Gilmore under the advice of Hans-Olav Tylli.

\bibliographystyle{abbrv}

\begin{thebibliography}{10}

\bibitem{Ahl78}
L.~V. Ahlfors.
\newblock {\em Complex analysis}.
\newblock McGraw-Hill Book Co., New York, third edition, 1978.

\bibitem{AA98a}
M.~P. Aldred and D.~H. Armitage.
\newblock Harmonic analogues of {G}. {R}. {M}ac{L}ane's universal functions.
\newblock {\em J. London Math. Soc. (2)}, 57(1):148--156, 1998.

\bibitem{AA98b}
M.~P. Aldred and D.~H. Armitage.
\newblock Harmonic analogues of {G}. {R}. {M}ac {L}ane's universal functions.
  {II}.
\newblock {\em J. Math. Anal. Appl.}, 220(1):382--395, 1998.

\bibitem{AG01}
D.~H. Armitage and S.~J. Gardiner.
\newblock {\em Classical potential theory}.
\newblock Springer Monographs in Mathematics. Springer-Verlag, London, 2001.

\bibitem{ABR01}
S.~Axler, P.~Bourdon, and W.~Ramey.
\newblock {\em Harmonic function theory}, volume 137 of {\em Graduate Texts in
  Mathematics}.
\newblock Springer-Verlag, New York, second edition, 2001.

\bibitem{BG06}
F.~Bayart and S.~Grivaux.
\newblock Frequently hypercyclic operators.
\newblock {\em Trans. Amer. Math. Soc.}, 358(11):5083--5117, 2006.

\bibitem{BM09}
F.~Bayart and {\'E}.~Matheron.
\newblock {\em Dynamics of linear operators}, volume 179 of {\em Cambridge
  Tracts in Mathematics}.
\newblock Cambridge University Press, Cambridge, 2009.

\bibitem{BBGE10}
O.~Blasco, A.~Bonilla, and K.-G. Grosse-Erdmann.
\newblock Rate of growth of frequently hypercyclic functions.
\newblock {\em Proc. Edinb. Math. Soc. (2)}, 53(1):39--59, 2010.

\bibitem{BB13}
J.~Bonet and A.~Bonilla.
\newblock Chaos of the differentiation operator on weighted {B}anach spaces of
  entire functions.
\newblock {\em Complex Anal. Oper. Theory}, 7(1):33--42, 2013.

\bibitem{BC55}
M.~Brelot and G.~Choquet.
\newblock Polyn\^omes harmoniques et polyharmoniques.
\newblock In {\em Second colloque sur les \'equations aux d\'eriv\'ees
  partielles, {B}ruxelles, 1954}, pages 45--66. Georges Thone, Li\`ege; Masson
  \& Cie, Paris, 1955.

\bibitem{DS12}
D.~Drasin and E.~Saksman.
\newblock Optimal growth of entire functions frequently hypercyclic for the
  differentiation operator.
\newblock {\em J. Funct. Anal.}, 263(11):3674--3688, 2012.

\bibitem{GEP11}
K.-G. Grosse-Erdmann and A.~Peris~Manguillot.
\newblock {\em Linear chaos}.
\newblock Universitext. Springer, London, 2011.

\bibitem{Kur71}
{\"U}.~Kuran.
\newblock On {B}relot-{C}hoquet axial polynomials.
\newblock {\em J. London Math. Soc. (2)}, 4:15--26, 1971.

\end{thebibliography}

\end{document}